\documentclass[a4paper]{amsart}
\usepackage[latin9]{inputenc}
\usepackage{a4}
\usepackage{amsfonts}
\usepackage{amssymb}
\usepackage{amsmath}
\usepackage{amsthm}
\usepackage{changepage}
\usepackage{nomencl}
\usepackage{geometry}
\usepackage{tikz}
\usepackage{verbatim}
\usetikzlibrary{matrix}
\usepackage{hyperref}
\usepackage[toc,page]{appendix}
\usepackage{mathrsfs}
\usepackage[square, numbers, comma]{natbib}

\usepackage{enumerate}
\usepackage{amsbsy}

\usepackage[all]{xy}
\usepackage{pb-diagram,pb-xy}
   \bibpunct{[}{]}{,}{n}{}{,}
\input cyracc.def

\begin{document}
\providecommand{\keywords}[1]{\textbf{\textit{Keywords: }} #1}
\newtheorem{theorem}{Theorem}[section]
\newtheorem{lemma}[theorem]{Lemma}
\newtheorem{proposition}[theorem]{Proposition}
\newtheorem{corollary}[theorem]{Corollary}
\newtheorem{problem}[theorem]{Problem}
\newtheorem{question}[theorem]{Question}
\newtheorem{conjecture}[theorem]{Conjecture}
\newtheorem{claim}[theorem]{Claim}
\newtheorem{condition}[theorem]{Condition}

\theoremstyle{definition}
\newtheorem{definition}[theorem]{Definition} 
\theoremstyle{remark}
\newtheorem{remark}[theorem]{Remark}
\newtheorem{example}[theorem]{Example}
\newtheorem{condenum}{Condition}

\def\p{\mathfrak{p}}
\def\q{\mathfrak{q}}
\def\s{\mathfrak{S}}
\def\Gal{\mathrm{Gal}}
\def\Ker{\mathrm{Ker}}
\def\soc{\mathrm{soc}}
\def\Coker{\mathrm{Coker}}
\newcommand{\cc}{{\mathbb{C}}}   
\newcommand{\ff}{{\mathbb{F}}}  
\newcommand{\nn}{{\mathbb{N}}}   
\newcommand{\qq}{{\mathbb{Q}}}  
\newcommand{\rr}{{\mathbb{R}}}   
\newcommand{\zz}{{\mathbb{Z}}}  
\def\K{\kappa}

\title{On the set of stable primes for postcritically infinite maps over number fields}
\author{Joachim K\"onig}
\address{Department of Mathematics Education, Korea National University of Education, Cheongju 28173, South Korea}
\email{jkoenig@knue.ac.kr}
\begin{abstract}
Many interesting questions in arithmetic dynamics revolve, in one way or another, around the (local and/or global) reducibility behavior of iterates of a polynomial. We show that for very general families of integer polynomials $f$ (and, more generally, rational functions over number fields), the set of stable primes, i.e., primes modulo which all iterates of $f$ are irreducible, is a density zero set. Compared to previous results, our families cover a much wider ground, and in particular apply to $100\%$ of polynomials of any given odd degree, thus adding evidence to the conjecture that polynomials with a ``large" set of stable primes are necessarily of a very specific shape, and in particular are necessarily postcritically finite.
\end{abstract}
\maketitle

\section{Introduction and statement of main results}
\label{sec:intro}

  In all of the following, let $K$ denote a number field.
  
  Let $f$ be a polynomial over $K$, and $a\in K$. Then $(f,a)$ is called {\it stable} if $f^{\circ n}(X)-a$ is irreducible for all $n\in \mathbb{N}$, where $f^{\circ n}$ denotes the $n$-th iterate of $f$. Similarly, if $p$ is a prime ideal of $O_K$, $(f,a)$ is called {\it stable modulo $p$}, if the mod-$p$ reduction of $f^{\circ n}(X)-a$ is defined, of non-decreasing degree and irreducible for all $n$. In this case, $p$ is also called a stable prime for $(f,a)$.  The same notions generalize to rational functions $f\in K(X)$ upon replacing $f^{\circ n}(X)-a$ by its numerator. In case $a=0$, we simply say ``$f$ is stable" instead of ``$(f,0)$ is stable".
  Stability (and its generalizations such as ``eventual stability") has been a key object of study in arithmetic dynamics, see, e.g., \cite{JL}, \cite{DHJMSS} or \cite{LSS} for global stability questions and \cite{BJ}, \cite{Jones}, \cite{Ferr} or \cite{MOS} for mod-$p$ stability.
  
  Although there are known examples of polynomials with positive density sets of stable primes (see, e.g., \cite[Theorem 6.1(2)]{Jones} for a sufficient criterion in the case of quadratic PCF maps, and \cite[Section 1]{MOS} for some explicit examples), several  previous papers (e.g., \cite{Ferr}, \cite{KNR}) contributed to the suspicion that ``usually" (i.e., for sufficiently general polynomials $f$), mod-$p$ stability should be a rare event, i.e., should occur only for a density zero set of all primes, or even only for finitely many primes. Cf.\ also Question 19.12 in \cite{Survey}.
  In \cite{Koe24}, the author showed, using purely group-theoretical methods, that indeed a positive density set of stable primes for $(f,a)$ forces $f$ to have an unusually small arboreal representation over $a$, i.e., the Galois group $G_{a,n}:=\textrm{Gal}(f^{\circ n}(X)-a/K)$ must be of order asymptotically at most the order of some iterated wreath product of {\it cyclic} groups. On the other hand, it is generally expected that the group $G_{a,n}$ is ``usually" close to (the maximal possible group, namely) an iterated wreath product of symmetric groups. This expectation can be traced back to work of Odoni \cite{Odoni2}, showing that for a generic polynomial (i.e., with coefficients viewed as independent transcendentals), the group $G_{a,n}$ is the full iterated wreath product for all $n$. 
 Later works such as \cite{Jones2} (for the quadratic case) and \cite{BDGHT} (for very general maps) developed various ``finite index conjectures" stating that the group $G_{a,\infty}:=\varprojlim_n G_{a,n}$ should be of finite index in $G_{t,\infty}$ (where $t$ denotes a transcendental taking the place of $a$), unless the pair $(f,a)$ falls into certain very explicit exceptional cases.
    Examples of arboreal representations whose Galois group $G_{a,\infty}$ is maximal possible are known, e.g., from \cite{Kadets}, with special cases going back to Odoni \cite{Odoni}, but these are still far from proving the general expectations.

   The paper \cite{Koe24} also showed
  that, {\it conditional} on the expectations on the largeness of arboreal representations as presented in \cite{BDGHT}, only compositions of linear translates of monomials and Dickson polynomials can have a positive density set of stable primes, but gave no unconditional results in this direction.
  Similar density zero conclusions under more restrictive ``largeness" assumptions 
  had been known from \cite{Ferr} or \cite{KNR}. 
    
  In the present paper, we combine the group-theoretical results of \cite{Koe24} with arithmetic techniques to show {\it unconditionally} that, indeed, for many (and even ``most", in a suitable sense) polynomials $f$, such unusually small representations cannot possibly exist for any value $a\in K$.

  Recall that a rational function $f$ over a field $K$ is called {\it postcritically finite} (PCF) if all critical points $\alpha$ of $f$ have a finite forward orbit $\{f^{\circ n}(\alpha)\mid n\in \mathbb{N}\}$, and {\it postcritically infinite} otherwise. A point with a finite (resp., an infinite) forward orbit is also called a {\it preperiodic} point (resp., a {\it wandering} point) of $f$. By the {\it multiplicity} of a point $\alpha\in \overline{K}\cup \{\infty\}$ under a rational function $f$ we always mean its multiplicity as a root of $f(X)-f(\alpha)$ (resp., of $1/f$ in case $f(\alpha)=\infty$). In particular, a critical point is a point of multiplicity $\ge 2$. To avoid confusion, the reader should be aware that this multiplicity is larger by $1$ than what is often called the multiplicity of a critical point (namely, its multiplicity as a root of $f'(X)$).
  
  We will show that, under surprisingly weak extra assumptions, postcritically infinite rational functions cannot possibly possess positive density sets of stable primes. 
  Our main result is the following.
  \begin{theorem}
  \label{thm:main}
  Let $f\in K(X)$ be a postcritically infinite rational function fulfilling the following: There exists some critical point $\alpha$ of $f$ which is a wandering point and whose multiplicity under $f$ is divisible by some prime number coprime to $\deg(f)$.  
  Then the set of stable primes of $f$ is of density zero.
  \end{theorem}

An important aspect of Theorem \ref{thm:main} arises due to the following easy observation: let $a\in K$ be arbitrary and $f_a(X):=(X-a)\circ f\circ (X+a)$ the conjugate of $f$ by the linear shift $X\mapsto X+a$. Then $f_a^{\circ n}(X) = f^{\circ n}(Y) - a$ where $Y:=X+a$. Therefore (mod-$p$) stability of $f_a$ is equivalent to (mod-$p$) stability of $(f,a)$. Moreover, via subtracting $a$ from each element, the orbit of a point $\alpha\in K$ under $f$ corresponds one-by-one to the one of $\alpha-a$ under $f_a$, yielding in particular natural correspondences of preperiodic points, wandering points, etc. This means that, whenever Theorem \ref{thm:main} applies to a function $f$, it automatically also applies to $f_a$ for any $a\in K$, yielding the same density-zero conclusion for all $(f,a)$, $a\in K$, simultaneously. This is important to keep in mind, since many problems in arithmetic dynamics, such as the aforementioned ``finite index" expectations, are known to be very difficult to solve simultaneously for all fibers of a given map, and even exhibiting infinitely many ``good" values $a$ for a given $f$ is challenging. 

Previously, another sufficient condition for density zero conclusions depending only on the multiplicities of wandering critical points, and thus applicable for all specialization values $a$ simultaneously, has occurred in \cite[Theorem 1.2]{MOS} and in strengthened form in \cite{OS}.  The clear upshot of Theorem \ref{thm:main} is that it requires information about the multiplicity of only {\it one} such critical point, not about all of them at once. 

Theorem \ref{thm:main} in particular applies to all postcritically infinite polynomials of prime degree, apart from unicritical polynomials. The latter can however be dealt with separately, and in fact, for unicritical polynomials we achieve a strict dichotomy between the postcritically finite and infinite case via the following.
\begin{theorem}
\label{thm:unicrit}
Let $f\in K[X]$ be a unicritical polynomial, i.e., $f=\lambda\circ X^d\circ \mu$ for linear polynomials $\lambda,\mu\in K[X]$. Then the following hold:
\begin{itemize}
\item[1)] If $f$ is postcritically infinite, then the set of stable primes of $(f,a)$ is of density zero, for all $a\in K$.
\item[2)] If $f$ is postcritically finite, then the set of stable primes of $(f,a)$ is of positive density, for all $a\in K$ outside of a ``thin set in the sense of Serre".\footnote{I.e., the union of a finite set and finitely many value sets $f_i(X_i(K))$ under morphisms $f_i: X_i\to \mathbb{P}^1_K$ ($i=1,\dots, n$) of degree $\ge 2$, with irreducible curves $X_i$ over $K$; see \cite[Definition 3.1.1]{Serre}.}
\end{itemize}
\end{theorem}
Note that Assertion 2) yields positive density of stable primes in ``most" fibers of $f$, since thin sets are small in the sense that they contain $0\%$ of all points of $K$ when counted by height; see, e.g., \cite[Proposition 3.4.2]{Serre}.

Combining Theorems \ref{thm:main} and Assertion 1) of \ref{thm:unicrit}, we then have the following corollary.
  \begin{corollary} 
  \label{thm:primedeg}
  Let $f\in K[X]$ be a postcritically infinite polynomial of prime degree. 
  Then the set of stable primes of $f$ is of density zero.
  \end{corollary}
  
The following consequence of Theorem \ref{thm:main} applies to ``almost all" polynomials (in a density sense) of any given odd degree. 
\begin{corollary}
\label{cor:odddeg}
Let $f\in K[X]$ be such that $f' = h(X)^2g(X)$ with $g\in K[X]$ squarefree of even degree $\deg(g)>0$, and such that at least one root of $g$ is a wandering point of $f$.
Then the set of stable primes of $f$ is of density zero.\\
In particular, if $f\in K[X]$ is a postcritically infinite polynomial of odd degree such that either
\begin{itemize}
\item[i)] $f'\in K[X]$ is squarefree, or 
\item[ii)] $f'$ is a non-square in $\overline{K}[X]$ and no finite critical point is preperiodic,
\end{itemize}
then the set of stable primes of $f$ is of density zero.
\end{corollary}
The above result should be compared with \cite[Theorem 1.2]{MOS} (see also \cite[Theorem 1.2]{OS} for an improvement of the asymptotic) which treats the case $\deg(g)=1$ (in which case, $f$ is necessarily of even degree). Note, however, that $\deg(g)=1$ is a ``nongeneric" assumption unless $\deg(f)=2$, since a generic polynomial has $g=f'$.

The following result about (among others) trinomials should be compared with \cite[Corollary 1.3]{MOS}, which treats (although with different methods) the situation analogous to the special case $k=d-1$ of part b) for {\it even} degree $d$.
\begin{corollary}
\label{cor:trinom}
The set of stable primes of $f$ is of density zero in each of the following two cases.
\begin{itemize}
\item[a)]
$f(X)=g(X)\cdot (X-a)^k + b$ for some polynomial $g\in K[X]$ and $a,b\in K$ such that $a$ is a wandering point of $f$, $g(a)\ne 0$, and $k$ is divisible by a prime coprime to $\deg(g)$. 
\item[b)]
$f(X)=r(X-a)^d+s(X-a)^{k}+b\in K[X]$, with $a,b,r,s\in K$, $r,s\ne 0$, such that $d$ is odd, $1<k<d$ is coprime to $d$, and 
at least one of 
the points $\gamma_0:=a$ and $\gamma_i:= a+\zeta_{d-k}^i \cdot \sqrt[d-k]{\frac{-ks}{dr}}$, $i=1,\dots, d-k$ 
is a wandering point of $f$. 
\end{itemize}
\end{corollary}


  The basic underlying idea for the proof of our main results is the observation that if a polynomial $f$ (with leading coefficient not divisible by $p$) is irreducible modulo a prime $p$, then $p$ is unramified in the splitting field of $f$ with Frobenius acting on the roots of $f$ as a ``full cycle" (i.e., a cycle of length the degree of the polynomial). On the other hand, Chebotarev's density theorem asserts that the proportion of such cycles in the Galois group of $f$ equals the density of primes $p$ possessing such a cycle as Frobenius. In particular, if one can show that the proportion of full cycles in $G_{a,n}:=\textrm{Gal}(f^{\circ n}(X)-a/K)$ converges to zero as $n\to \infty$, then the proportion of primes modulo which $f^{\circ n}(X)-a$ is irreducible must converge to zero as $n\to \infty$, implying in particular that the density of stable primes of $(f,a)$ is zero. The main group-theoretical tools to derive such density zero results are recalled in Section \ref{sec:stabpr_gp}. Notably, the action of the Galois group $G_{a,n}$ on roots of $f^{\circ n}(X)-a$ is imprimitive.
  Proposition \ref{prop:stabpr_gp} asserts a shrinking of the proportion of full cycles in such imprimitive groups (when compared with the action on blocks), unless the group itself, and notably the block kernel, is of a very special shape. Corollary \ref{cor:stabpr} then concludes that (the proportion of full cycles in $G_{a,n}$ converges to zero as $n\to \infty$, hence) the set of stable primes is of density zero, as long as, for infinitely many $n$, there exist certain group elements which prevent the aforementioned special shape. The emphasis on individual group elements here is due to the next ingredient of the proof.
   Namely, to show that the iterated Galois groups $G_{a,n}$ fulfill the required conditions, we use the inertia group generators at ramified primes in the splitting fields of $f^{\circ n}(X)-a$. Controlling these inertia groups requires certain results about ramification in specializations of function field extensions (Theorem \ref{thm:spec_in}), as well as about the prime divisors of dynamical sequences (Lemma \ref{lem:betal}).
   
  After this general setup, our main results will be derived in Section \ref{sec:proofs}.  
The potential of our method however goes beyond the results stated so far. A strong technical result, useful for generalizing Theorem \ref{thm:main}, is contained in Theorem \ref{thm:tech_stronger}. While translating this theorem one-by-one into a property of polynomials in the style of Theorem \ref{thm:main} may be difficult, we give a sample application, which is not covered yet by Theorem \ref{thm:main} (and in particular generalizes Corollary \ref{cor:trinom}a)).
\begin{theorem}
\label{thm:main_extra}
Let $f(X) = g(X)(X-a)^k+b \in K[X]$, where $a,b\in K$, $g\in K[X]$ is separable with $g(a)\ne 0$, $a$ is a wandering point of $f$ and $k$ is an odd integer not dividing $\deg(f)$. Then the set of stable primes of $f$ is of density zero.
\end{theorem}

  For some further applications even to certain cases where all the multiplicities under $f$ of critical points divide the degree, see Lemma \ref{lem:more_ex}. In fact, we conjecture:

  \begin{conjecture}
  \label{conj:bold}
  For every postcritically infinite polynomial $f\in K[X]$ over a number field $K$, the set of stable primes is of density zero.
  \end{conjecture}

In the appendix, we give some further motivation for Conjecture \ref{conj:bold} (Appendix \ref{sec:large}), and also include some thoughts on the {\it maximal} possible density of stable primes that a polynomial can have (Appendix \ref{sec:lowdeg}).
  
  We finally remark that the postcritically {\it finite} case, while not the main focus of this paper, also deserves further attention. As evidenced by Assertion 2) of Theorem \ref{thm:unicrit}, there are indeed examples of PCF polynomials with positive density set of stable primes. On the other hand, there are also some density zero results for stable primes of $(f,a)$ for certain PCF polynomials $f$, see, e.g., Proposition 6.1 in combination with Proposition 5.1 of \cite{KNR}. In that case, however, it does not seem feasible (with the methods known to the author) to treat, for a given $f$, all values $a\in K$ at the same time. 
  
  {\bf Acknowledgement}: I thank the two referees for their very detailed and helpful reports. This work was supported by the 2025 Sabbatical Leave Research Grant funded by Korea National University of Education.

  \section{Prerequisites}
 \subsection{Some basics of imprimitive permutation groups}
 \label{sec:prelim_wreath}
Let $f=f_1\circ \dots\circ f_r$ be an irreducible polynomial over a number field $K$, arising as a composition of non-linear polynomials $f_i\in K[X]$, and let $G:=\textrm{Gal}(f/K)$. It is well-known that $G$ acts transitively, but imprimitively, on the roots of $f$, and more precisely embeds as a permutation group into the iterated (imprimitive) wreath product $G_r\wr\dots\wr G_1 \le S_d$, where $d:=\deg(f)$ and $G_i:=\textrm{Gal}(f_i(X)-t/K(t))$  (with a transcendental $t$). 
%

In view of the results recalled in Section \ref{sec:stabpr_gp} (and their application in Section \ref{sec:proofs}), it will be useful to investigate to some extent the possible cycle structures of elements in iterated wreath products. For this, let $H_1\le S_{d_1}$, $\dots$, $H_r\le S_{d_r}$ be transitive permutation groups and $\sigma\in H_r\wr\dots\wr H_{1}\le S_{d_1\cdots d_r}$. Successively projecting $\sigma$ to its image in $H_i\wr\dots\wr H_1$, $i=1,\dots, r$, yields a sequence of partitions $\mathcal{S}_i:=(\lambda_{i,1},\dots, \lambda_{i,N_i})$ for each $i\in \{1,\dots, r\}$, with every $\lambda_{i,j}$ corresponding to a cycle type in $H_i$; 
and furthermore for each $i$ a bijection $\{1,\dots, N_i\}\to \{(a,b)\mid 1\le a\le N_{i-1}, 1\le b\le \#\lambda_{i-1,a}\}$, identifying the unique partition on the $i$-th level which is used to extend the $b$-th entry of the $a$-th partition on the $i-1$-th level.
This fact may be seen purely group-theoretically, but from a number-theoretical point of view may be most conveniently explained as follows: Let $F_0\subset F_1\subset\dots\subset F_r$ be a chain of number fields, with Galois groups of the Galois closure of $F_i/F_{i-1}$ permutation-isomorphic to $H_i$, and let $p$ be a prime of $F_0$ unramified in $F_r$ with Frobenius equal to the conjugacy class of $\sigma$.\footnote{This translation does not lose any generality, since every finite group occurs as a Galois group over some number field $F_0$, and each of its conjugacy classes occurs as the Frobenius at a suitable prime by Chebotarev's density theorem.}
Then the above system of partitions and maps is nothing but the rooted tree of prime ideals extending $p$ in the fields $F_i$, $i=1,\dots, r$, with the partitions taking note of the respective relative degrees in each step. Via running through the paths from the root to the leaves in this tree, every cycle of $\sigma$ is associated to a unique sequence $((a_1,b_1),\dots, (a_r,b_r))$ with $a_i\in \{1,\dots, N_i\}$ and $1\le b_r\le \#\lambda_{i,a_i}$ identifying the elements of $\mathcal{S}_i$, $i=1,\dots, r$, used to compose the respective cycle.
 
Compare also the notion of the ``portrait" of $\sigma$, which is used, mainly in a context where the wreath product is interpreted as a group of automorphisms of a rooted tree, to fully describe $\sigma$ by successively identifying the induced permutations on each level of the tree (e.g., \cite{Bartholdi}), similar to the successive reconstruction of cycle types outlined above.
We will come back to the above setup in the proof of Theorem \ref{thm:main_extra}.
  
  \subsection{Previous results on mod-$p$ stability}
  \label{sec:stabpr_gp}
  The crucial group-theoretical tools for deducing our main results are developed in \cite{Koe24}. We first give them in a purely group-theoretical form (cases a) and b) of the following proposition), before deriving the immediate application to stable primes.
  Case c) of the proposition is comparatively much more elementary (see, e.g., \cite{KNR} for a broad generalization of the underlying idea), but nevertheless useful. To understand the relevance of the setup of Proposition \ref{prop:stabpr_gp}, note that imprimitive permutation groups $G$ such as the Galois groups of iterates $f^{\circ n}(X)-a$ (see Section \ref{sec:prelim_wreath}) are characterized by the point stabilizers $U$ being non-maximal subgroups, cf. \cite[Corollary 1.5A]{DM}, with proper overgroups $V>U$ acting as block stabilizers.  

  \begin{proposition}
  \label{prop:stabpr_gp}
  Let $G$ be a group, $U<G$ a proper subgroup of finite index, and let $V>U$ be a minimal overgroup. Let $G_U$ and $G_V$ be the image of $G$ in the action on cosets of $U$ and of $V$,  respectively, and let $N$ be the kernel of the natural projection $G_U\to G_V$. Finally, let $H$ be the image of $V$ in the action of cosets of $U$ in $V$, and let $c_U$ (resp., $c_V$) $\in [0,1)$ be the proportion of elements of $G_U$ (resp., of $G_V$) acting as ``full cycles" (i.e., elements generating cyclic transitive subgroups). Then in each of the following three cases, there is a constant $\delta<1$, depending only on $[V:U]$ and on the respective case, such that $c_U \le \delta \cdot c_V$:
  \begin{itemize}
  \item[a)] $H$ does not embed (as a transitive subgroup) into $AGL_1(p)$ for any prime $p$, where $AGL_1(p)=C_p\rtimes C_{p-1}$ denotes the normalizer of a $p$-Sylow group in $S_p$.
  \item[b)] $H\le AGL_1(p)$ for some prime $p$, but $N$ is not elementary-abelian.
  \item[c)] $N$ is elementary-abelian of (maximal possible) order $p^{[G:V]}$ (where $p=[V:U]$ is a prime).
  \end{itemize}
  \end{proposition}
  \begin{proof}
  Upon noting that every image of a full cycle in $G_U$ must be a full cycle in $G_V$, it becomes obvious that one may choose $\delta$ as the maximal proportion of full cycles in any $N$-coset of a full cycle in $G_U$. The conclusions in cases a) and b) are then straightforward from assertions 2) and 1), respectively, of \cite[Proposition 3.2]{Koe24}. Indeed, as shown there, for case a) the constant $\delta$ can be chosen independently of $[V:U]$, and $\delta=\frac{1}{2}$ is an admissible choice, whereas for case b), $\delta = \frac{p-1}{p}$ works (note that here, one necessarily has $p=[V:U]$).  Finally, condition c) translates to saying that the embedding $G_U\le AGL_1(p)\wr G_V$ induced by the chain $G>V>U$ has ``large kernel" in the sense of \cite[\S 2.3]{KNR}. It is elementary to show that in this case, every full cycle of $G_V$ lifts to at least one element of $G_U$ which is {\it not} a full cycle (see, e.g., the proof of \cite[Proposition 5.1]{KNR} for a more general version of this), showing already $c_U<c_V$. On the other hand, if $C\subset G_V$ is a conjugacy class of full cycles, one has $|C|=\frac{|G_V|}{[G:V]}$, and similarly any conjugacy class of full cycles of $G_U$ in the preimage of $C$ is of length $\frac{|G_U|}{[G:U]} = \frac{|G_V|\cdot |N|}{[G:V]\cdot [V:U]}=|C|\cdot \frac{|N|}{p}$, meaning that as soon as the preimage of $C$ is not entirely composed of full cycles, at least a proportion of $\frac{1}{p}$ of this set are not full cycles. The assertion therefore holds with $\delta = \frac{p-1}{p}$ in this case.  
  \end{proof}

  We reword the most important (for us) consequences of Proposition \ref{prop:stabpr_gp} in a way suited specifically for deducing density zero results as in Section \ref{sec:proofs}.
  See also Theorems 1.2 and 3.3 in \cite{Koe24} for this kind of translation into an arithmetic dynamics setting.
    \begin{corollary}
  \label{cor:stabpr}
  Let $f\in K(X)$ be a rational function  and $a\in K$. Denote by $\kappa_n$ a field generated by a root of $f^{\circ n}(X)-a$ over $K$, with the roots picked such that $\kappa_n\subset \kappa_{n+1}$ for all $n$, and by $K_n$ the splitting field of $f^{\circ n}(X)-a$. Then the set of stable primes of $(f,a)$ is of density zero as soon as there are infinitely many $n\in \mathbb{N}$ for which one of the following is satisfied:
  \begin{itemize}
 \item[a)]
 The group $\mathrm{Gal}(K_n/K_{n-1})$ contains a non-identity element of order coprime to $\deg(f)$.
 \item[b)] The Galois group of the Galois closure of $\kappa_n/\kappa_{n-1}$ contains an element of a cycle type which is not contained in any iterated wreath product $AGL_1(p_r)\wr\dots \wr AGL_1(p_1)$ for primes $p_1,\dots, p_r$ with $\prod_{i=1}^r p_i=\deg(f)$.
  \end{itemize}
  \end{corollary}
  \begin{proof} 
 Recall first that, for any $n\in \mathbb{N}$ and any prime $p$ of $K$, the mod-$p$ irreducibility of $f^{\circ n}(X)-a$ implies that $p$ is unramified in $K_n$ with Frobenius acting as a full cycle. Hence, by Chebotarev's density theorem, the assertion follows as soon as we show that the proportion of full cycles in $G_{a,n}:=\textrm{Gal}(f^{\circ n}(X)-a/K)$ converges to zero as $n\to \infty$.
  
  Refine the extensions $\kappa_{n}/\kappa_{n-1}$ to a sequence $\kappa_{n-1}=:F_{n,0}\subset F_{n,1}\subset\dots\subset F_{n,r}:=\kappa_n$ such that each $F_{n,i-1}$ is a maximal subfield of $F_{n,i}$, and for each $i\le r (=r(n))$, let $\Omega_{n,i}$ be the Galois closure of $F_{n,i}/K$. Note that we may assume $[K_n:K]=\deg(f)^n$ (or else, $f^{\circ n}(X)-a$ is reducible, whence there are no stable primes) and hence $[F_{n,i}:F_{n,0}]\mid \deg(f)$ for all indices $n$ and $i$. 

We will choose a group $G$, and subgroups $U$, $V$, $N$ with the goal of applying Proposition \ref{prop:stabpr_gp}.
 Since this choice depends on whether Condition a) or b) of Corollary \ref{cor:stabpr} holds, 
assume first that Condition a) holds for infinitely many $n$. For $B\ge 0$ arbitrary, pick the minimal integer $j\in \mathbb{N}$ such a) holds for at least $B$ different indices $n\le j$. 
  Since, by assumption, $\mathrm{Gal}(K_j/K_{j-1}) = \mathrm{Gal}(\Omega_{j,r}/\Omega_{j,0})$ contains a non-identity element of order coprime to $\deg(f)$, there exists a minimal index $i\in \{1,\dots,r\}$ such that 
  $\mathrm{Gal}(\Omega_{j,i}/\Omega_{j,0})$ contains a non-identity element of order coprime to $[F_{j,i}:F_{j,0}]$. 
  Set $G:=\textrm{Gal}(\Omega_{j,i}/K)$, $U:=\textrm{Gal}(\Omega_{j,i}/F_{j,i})$ and $V:=\textrm{Gal}(\Omega_{j,i}/F_{j,i-1})$, as well as $N:=\mathrm{Gal}(\Omega_{j,i}/\Omega_{j,i-1})$.
  Then $V$ is a minimal overgroup of $U$ by the maximality of $F_{j,i-1}$ in $F_{j,i}$ and the Galois correspondence, and $N$ is the kernel of the projection from $G_U$ to $G_V$ in the notation of Proposition \ref{prop:stabpr_gp}.
Moreover, $N$ cannot be elementary abelian (of exponent $[F_{j,i}:F_{j,i-1}]$), since otherwise every element of $\mathrm{Gal}(\Omega_{j,i}/\Omega_{j,0})$ would be of order dividing $[F_{j,i}:F_{j,i-1}]$ times an element order in $\mathrm{Gal}(\Omega_{j,i}/\Omega_{j,0})$, and hence not coprime to $[F_{j,i}:F_{j,i-1}]\cdot [F_{j,i-1}:F_{j,0}] = [F_{j,i}:F_{j,0}]$. We are therefore in Case a) or b) of Proposition \ref{prop:stabpr_gp} (with $G$, $U$, $V$ and $N$ as defined here).

Assume now on the other hand that Condition b) of Corollary \ref{cor:stabpr} holds for infinitely many $n$. For $B\ge 0$, pick the minimal integer $j\in \mathbb{N}$ such that b) holds for at least $B$ different indices $n\le j$. Since, by assumption, the Galois group of the Galois closure of $\kappa_n/\kappa_{n-1}$ does not embed into  any iterated wreath product $AGL_1(p_r)\wr\dots \wr AGL_1(p_1)$, there must be a minimal index $i\in \{1,\dots, r\}$ such that the Galois group $H$ of the Galois closure of $F_{j,i}/F_{j,i-1}$ does not embed (as a transitive subgroup) into any group $AGL_1(p)$.  Set again $G:=\textrm{Gal}(\Omega_{j,i}/K)$, $U:=\textrm{Gal}(\Omega_{j,i}/F_{j,i})$ and $V:=\textrm{Gal}(\Omega_{j,i}/F_{j,i-1})$, as well as $N:=\mathrm{Gal}(\Omega_{j,i}/\Omega_{j,i-1})$. We are therefore in Case a) of Proposition \ref{prop:stabpr_gp} (with $G$, $U$, $V$ and $N$ as defined here).

Altogether, we are always in case a) or b) of Proposition \ref{prop:stabpr_gp}.
Let $\delta_0$ be the maximum of the two constants $\delta$ from cases a) and b) of that proposition. 
We now claim that the proportion of full cycles in $G$ is upper bounded by $\delta_0^B$. Of course, the proportion of full cycles in $G_{a,n}$ is then also upper bounded by $\delta_0^B$ for all $n\ge j$, since full cycles of $G_{a,n}$ project to full cycles of $G$. The assertion thus follows from the claim, since obviously $\lim_{B\to \infty}\delta_0^B = 0$. We show the claim by induction over $B$. For $B=0$, there is nothing to show, so assume that the claim holds for $B-1$. In particular, the proportion of full cycles in the image $G_V$ of $G$ in the action on cosets of $V$ is upper bounded by $\delta_0^{B-1}$, so that, in the notation of Proposition \ref{prop:stabpr_gp}, we have $c_V\le \delta_0^{B-1}$. 
It thus follows from Proposition \ref{prop:stabpr_gp} that $c_U\le \delta_0\cdot c_V\le \delta_0^B$, which is the claim.
  \end{proof}

  \subsection{Ramification in specializations and prime divisors of dynamical sequences}
  The above gives a proof strategy to show that the proportion of stable primes of $(f,a)$ is zero in many concrete cases: it suffices to identify elements of suitable cycle structures in the kernel $\ker(\pi_n)$ for $\pi_n: \textrm{Gal}(f^{\circ n}(X)-a/K)\to \textrm{Gal}(f^{\circ n-1}(X)-a/K)$ the natural projection (and then use Corollary \ref{cor:stabpr}a)), or in the stabilizer of a root of $f^{\circ n-1}(X)-a$ (and then use Corollary \ref{cor:stabpr}b)), for infinitely many $n$. We do this by identifying inertia group generators of certain cycle types. To achieve this, we use a version of the ``specialization inertia theorem", linking ramification in the splitting field of $f^{\circ n}(X)-t$ over $K(t)$ with that in the splitting field of $f^{\circ n}(X)-a$ over $K$.    
  To state the theorem, recall the following notion: Let $R$ be a Dedekind domain with fraction field $K$, and $p$ a prime ideal of $R$. For $a\in R$ and $b$ an $R$-integral element of a finite extension of $K$,\footnote{More general definitions are possible (see, e.g., \cite{Leg}), but not necessary for our purposes and therefore omitted for sake of simplicity.} the {\it intersection multiplicity} of $a$ and $b$ at $p$ is defined as $m_p(a,b):=\nu_p(\mu_b(a))$, where $\mu_b\in R[X]$ denotes the minimal polynomial of $b$ and $\nu_p$ is the $p$-adic valuation.
  \begin{theorem}
  \label{thm:spec_in}
  Let $K$ be the fraction field of a Dedekind domain $R$ of characteristic $0$ such that all residue fields of $R$ modulo prime ideals are perfect. Let $F(t,X)\in K[t,X]$ be a monic separable polynomial, 
 $E/K(t)$ the splitting field of $F$ and $G:=\mathrm{Gal}(E/K(t))$. 
Let $t_1,\dots, t_r\in \overline{K}\cup \{\infty\}$ be the branch points of $E/K(t)$. 
  Let $\mathcal{S}_0:=\mathcal{S}_0(F,K)$ be the union of the following sets of prime ideals of $R$.
  \begin{itemize}
  \item[i)] Primes modulo which $F$ is not defined (i.e., dividing the denominator of some coefficient of $F$).
  \item[ii)] Primes dividing the discriminant $\Delta$ of $F(t,X)$ with respect to $X$. 
  \item[iii)] Primes $p$ such that mod-$p$ reduction of $\Delta$ decreases the number of roots of $\Delta$ (not counting multiplicity).\footnote{I.e., such that either two distinct factors of $\Delta$ coincide mod $p$ or some root of $\Delta$ is of negative $p$-adic valuation.}
  \item[iv)] Primes dividing $|G|$.
  \end{itemize}
  Call a prime $p$ of $R$ ``good" for $F$ if it does not belong to $\mathcal{S}_0$. Then for all good primes of $R$ and all $a\in R\setminus\{t_1,\dots, t_r\}$, the following hold:
  \begin{itemize}
  \item[a)] If $p$ ramifies in the residue field of $E/K(t)$ at $t\mapsto a$,\footnote{For values $a$ with $F(a,X)$ separable, this residue field is simply the splitting field of $F(a,X)$.} then there exists some branch point $t_i$ such that $m_p(a,t_i)>0$. 
  \item[b)] If conversely such $t_i$ exists, then the inertia group at $p$ in this residue field can be identified (up to conjugation in $G$) with a subgroup of the inertia group $I_{t_i}$ at (any place extending) $t\mapsto t_i$ in $E/K(t)$, and the ramification index at $p$ equals $\frac{|I_{t_i}|}{\textrm{gcd}(|I_{t_i}|, m_p(a,t_i))}$, where $m_p(a,t_i)$ denotes the intersection multiplicity of $a$ and $t_i$ at $p$.
  \end{itemize}  
  \end{theorem}
  \begin{proof}
  The conclusion is contained, e.g., in the Specialization Inertia Theorem stated in \ \cite[Section 2.2.2]{Leg}. 
  That paper also contains an explicit description of the finite set of exceptional primes, however slightly less convenient for us, since not connected to a concrete polynomial.
  To get this connection, one may additionally invoke \cite[Theorem 2.2]{KN_locdim}, which defines $\mathcal{S}_0$ as the set of primes fulfilling one of conditions ii), iii) or iv) above, but with the extra assumption that the coefficients of $F(t,X)$ are in $R$. Note that the latter theorem moves on from assertions about inertia groups to assertions about decomposition groups and residue degrees. Nevertheless, its set $\mathcal{S}_0$ can be taken for our purposes, for the following reason: As pointed out in the first sentence of the proof of  \cite[Theorem 2.2]{KN_locdim}, it is shown there that this set contains the set $\mathcal{S}_{exc}$ of (the intermediate result) \cite[Theorem 4.1]{KLN}, which in term is explicitly defined to contain the set of exceptional primes from the Specialization Inertia Theorem in \cite{Leg}, see \cite[Remark 2.3]{KN_locdim}. Therefore, the conclusion of the latter holds for all primes not in the exceptional set of \cite{KN_locdim}.
  It only remains to justify that we can allow $F(t,X)\in K[t,X]$ rather than $F(t,X)\in R[t,X]$, at the cost of additionally excluding the primes in i). This is however obvious, since for any given prime $p$ modulo which $F$ is defined, we can apply the theorem with the local ring $R_{(p)}$ at $p$ instead of $R$.
  \end{proof}
  
  The following is immediate from the explicit description of $\mathcal{S}_0$ given above.
  \begin{corollary}
  \label{cor:base_change}
  With the terminology of Theorem \ref{thm:spec_in}, let $L/K$ be a finite extension. Then a prime of $L$ is good for $F$ if it extends a prime of $K$ which is good for $F$.
  \end{corollary}
  
  Note that for the case of a rational function $f=\frac{f_1}{f_2}$ the branch points of the splitting field of $f_1(X)-tf_2(X)$ over $K(t)$ are simply the critical values of $f$, and the cycle type of the inertia group generator at a given branch point $t_i$ is the multiset of multiplicities of preimages of $t_i$ under $f$. Similarly, the branch points of the splitting field of (the numerator of) $f^{\circ n}(X)-t$ are all the iterates $f^{\circ k}(t_i)$ with $t_i$ a critical value of $f$ and $0\le k\le n-1$. An important point to remember when applying Theorem \ref{thm:spec_in} to arithmetic dynamics problems is that, for a postcritically infinite $f$, the set of these $f^{\circ k}(t_i)$ is of unbounded cardinality (as $n\to \infty$), while its elements are of course all of bounded degree over $K$, meaning that {\it every} prime will eventually fall into case iii) above, for $n\to \infty$. Hence application to the problems considered in this paper requires some care. 
  
  The following result on prime divisors in dynamical sequences will be needed,  in view of Theorem \ref{thm:spec_in}, to ensure divisibility of ramification indices by certain primes. It is a slight modification of \cite[Lemma 12]{Betal}. In fact, a) is explicitly contained there, but we will also require b), which from a purely logical point of view is slightly stronger.
  \begin{lemma}
  \label{lem:betal}
  Let $f\in K(X)$ be a rational function over a number field $K$ such that $0$ is not postcritical for $f$, and let $t_i\in K$ be a wandering point of $f$. Then, given any integer $e\ge 2$ and any finite set $\mathcal{S}$ of primes of $O_K$, the following hold:
  \begin{itemize}
  \item[a)] There is an infinite set of integers $n\in \mathbb{N}$ and for each such $n$ at least one prime $p$ of $O_K$ not contained in $\mathcal{S}$, such that the $p$-adic valuation of $f^{\circ n}(t_i)$ is positive and not divisible by $e$.
  \item[b)] More precisely, there is an infinite set $\mathcal{M}\subseteq \mathbb N$ such that for every $n\in \mathcal{M}$ we have $$\mathrm{gcd}(\{e\}\cup \{\nu_p(f^{\circ n}(t_i)): p\notin \mathcal{S} \text{ and } \nu_p(f^{\circ n}(t_i)) > 0\}) = 1.$$  \end{itemize}\end{lemma}
  \begin{proof}
 Part a) 
 is precisely \cite[Lemma 12]{Betal}.
 Part b) 
 is also immediate from the proof of Lemma 12 in \cite{Betal}, whose main ideas we thus sketch here: Write $f(X/Y)=\frac{F(X,Y)}{G(X,Y)}$ as a quotient of homogeneous polynomials, and 
 in order to avoid primes simultaneously dividing the numerator and denominator (in this representation) of some value $f(x/y)$, enlarge $\mathcal{S}$ to contain the finitely many prime divisors of the resultant $\textrm{Res}(F,G)$. Next, pick any prime divisor $d$ of $e$. 
 As noted in the proof of \cite[Lemma 12]{Betal}, one may invoke a result by Darmon and Granville (\cite[Theorem 1]{DG}), possibly upon replacing $f$ by a suitable iterate $f^{\circ k}$, ensuring that for any fixed $s\in K\setminus\{0\}$, the equation $sF(X,Y)=Z^d$ has only finitely many solutions $(x,y,z)$ in $K$ with the values $\frac{x}{y}$ distinct and the ideal $(x,y)\subseteq O_K$ dividing some fixed ideal. 
 Letting $(X,Y)$ run over the infinitely many values $(\alpha_n,\beta_n)\in O_K^2$, $n\in \mathbb{N}$ such that $f^{\circ n}(t_i) = \frac{\alpha_n}{\beta_n}$, and $s$ over a finite set of representatives of $\mathcal{S}$-units modulo $d$-th powers, it follows that, for all sufficiently large $n$, the $\mathrm{gcd}$ of valuations $\nu_p(F(\alpha_{n-1}, \beta_{n-1}))$, $p\notin \mathcal{S}$ must be coprime to $d$.  Combining this for all prime divisors $d$ of $e$, it follows that there exists $n\in \mathbb{N}$ such that $\mathrm{gcd}(\{e\}\cup \{\nu_p(f^{\circ n}(t_i)): p\notin \mathcal{S} \text{ and } \nu_p(f^{\circ n}(t_i)) > 0\}) = 1$.
 \end{proof}

  \section{Proofs of the main results}
  \label{sec:proofs}
 The proof of Theorem \ref{thm:main} will be achieved via the following technical result which sharpens \cite[Theorem 5]{Betal}.
  \begin{theorem}
 \label{thm:tech_strong}
 Let $f\in K(X)$, let $a\in K$ be neither postcritical nor a fixed point of $f$, and let $q$ be a prime number not dividing $\deg(f)$. Let  $G_n:=\mathrm{Gal}(f^{\circ n}(X)-a/K)$ and $\pi_n:G_n\to G_{n-1}$ the natural projection.  Assume that there exists at least one wandering critical point $t_i$ of $f$ whose multiplicity under $f$ is divisible by $q$. Then for infinitely many $n\in \mathbb{N}$, the kernel $\ker(\pi_n)\le G_n$ contains an element of order divisible by $q$.
 \end{theorem}

 \begin{proof}[Proof of Theorem \ref{thm:main} assuming Theorem \ref{thm:tech_strong}]
 This is now an immediate consequence of Corollary \ref{cor:stabpr}. Note only that the technical assumptions on $a$ in Theorem \ref{thm:tech_strong} are harmless, since postcritical values (resp., fixed points) of $f$ render some iterate $f^{\circ n}(X)-a$ inseparable (resp., render $f(X)-a$ reducible), so that $(f,a)$ would not even be stable in such a case.
 \end{proof}
  
In order to be able to derive Corollary \ref{thm:primedeg}, we first insert a proof of Theorem \ref{thm:unicrit}.

\begin{proof}[Proof of Theorem \ref{thm:unicrit}]
Begin by noting that up to conjugacy over $K$, i.e., replacing $(f,a)$ by $(\mu\circ f\circ \mu^{-1}, \mu(a))$  we may (and will) assume $f$ to be of the form $f(X)=uX^d+v$ for $u,v\in K$, $u\ne 0$.
First, assume that $f$ is postcritically infinite. Let $p$ be a prime of $K$ such that $(f,a)$ is stable mod $p$. It follows first (see, e.g., \cite[Remark 14]{DDetal}) that $p$ is of norm $N(p)\equiv 1 \pmod q$ for every prime number $q$ dividing $d$, and $N(p)\equiv 1 \pmod 4$ in case $4|d$; indeed, otherwise no polynomial of the form $X^q-\alpha$ (resp., $X^4-\alpha$) would be irreducible modulo $p$, whence $f(X)-a$ would be reducible modulo $p$.
Now, fix any prime divisor $q$ of $d$. From \cite[Theorem 13]{DDetal}, it is then necessary that $\frac{f^{\circ n}(v)-a}{u}$ is not a $q$-th power modulo $p$ for any $n\ge 1$. In other words (since the completion of $K$ at $p$ contains the $q$-th roots of unity due to $N(p)\equiv 1 \pmod q$), $p$ is inert in the extension $K_n:=K\left(\sqrt[q]{\frac{f^{\circ n}(v)-a}{u}}\right)$ of $K$, for all $n\ge 1$. But due to Lemma \ref{lem:betal} (applied with the divisor $q$ of $d=\deg(f)$), the set of primes of $K$ ramifying in at least one $K_n$ is infinite, whence infinitely many of the fields $K_n$ are linearly disjoint over $K(\zeta_q)$. This implies that the density of primes 
remaining inert in all of them is zero, ending the proof of 1).

Assume now that $f$ is postcritically finite. and let $n\in \mathbb{N}$ be such that the set $\{f^{\circ j}(0)\mid 1\le j\le n\}$ is the full forward orbit of the critical point $0$ of $f$. It follows again from \cite[Theorem 13]{DDetal} that stability of $(f,a)$ modulo $p$ is equivalent to irreducibility of $f^{\circ j}(X)-a \bmod p$ for all $j\le n+1$. 
Hence for the set of stable primes of $(f,a)$ to be of positive density, it suffices that $f^{\circ n}(X)-a$ is irreducible modulo a positive density set of primes $p$ of $K$. For all primes $p$ modulo which the latter polynomial is separable, its mod-$p$ irreducibility is equivalent 
to the Frobenius at $p$ acting as a generator of a cyclic transitive subgroup. Thus, by Chebotarev's density theorem, the set of stable primes must be of positive density, as long as $\textrm{Gal}(f^{\circ n}(X)-a/K)$ contains a cyclic transitive subgroup. Since the monodromy group $\textrm{Gal}(f^{\circ n}(X)-t/K(t))$ contains such a subgroup (namely, an inertia group at $t\mapsto \infty$), Hilbert's irreducibility theorem 
yields a positive density set of stable primes $p$ for all $a$ outside of a thin set.
\end{proof}

\begin{remark}
Theorem \ref{thm:unicrit} and its proof also allow to determine the precise (positive) proportion of stable primes for certain unicritical PCF polynomials. As an example, take $f=1-X^d\in \mathbb{Q}[X]$, which has critical orbit $\{0,1\}$. To simplify matters, we will assume $d$ to be an odd prime. By \cite[Theorem 13]{DDetal} as already used in the above proof, mod-$p$ stability of $(f,a)$ is equivalent to $f^{\circ j}(X)-a$ being irreducible modulo $p$ for all $j\le 3$, which in this case amounts to none of $1-a$, $a$ and $a-1$ being a $d$-th power modulo $p$; since $d$ is odd, the last of the three conditions is already implied by the first, so that mod-$p$ stability is in fact equivalent to $f(f(X))-a$ being irreducible modulo $p$. In particular,
the proportion of stable primes of $(f,a)$ equals the proportion of $d^2$-cycles in $G_{a,2}:=\textrm{Gal}(f(f(X))-a/\mathbb{Q})$.
Since $d$ is a prime, all $d$-cycles in $G_{a,1}:=\textrm{Gal}(f(X)-a/\mathbb{Q})$ (and a fortiori, all $d^2$-cycles in $G_{a,2}$) lie in the index-$(d-1)$ subgroup fixing $d$-th roots of unity. Since it is easy to see (via considering the ramification type of the map $X\mapsto f(f(X))$) that $\textrm{Gal}(f(f(X))-t/\mathbb{Q}(\zeta_d)(t))\cong C_d\wr C_d$, in which the proportion of $d^2$-cycles equals $(\frac{d-1}{d})^2$, it follows that for ``most" $a\in \mathbb{Q}$ (namely, all outside of a thin set), the proportion of stable primes of $(f,a)$ equals $\frac{1}{d-1}\cdot (\frac{d-1}{d})^2=\frac{d-1}{d^2}$. For certain values $a$, a higher proportion can be reached, e.g., for $d=3$, a computation with Magma (investigating the proportion of $9$-cycles in maximal subgroups of $C_3\wr C_3$ and then parameterizing the corresponding fixed fields) yielded that for $a\in \frac{1}{1-X^3}(\mathbb{Q})$ (exempting those $a$ rendering $f(f(X))-a$ reducible), the Galois group $G_{a,2}$ shrinks to a subgroup with $9$-cycle proportion $\frac{1}{3}$, whence the density of stable primes increases to $\frac{1}{3}$.
\end{remark}

 We continue with the proof of the corollaries stated in Section \ref{sec:intro}.
 \begin{proof}[Proof of Corollary \ref{thm:primedeg}]
 For a polynomial of prime degree $p$, infinity is automatically a critical point of multiplicity $p$ under $f$. If there were a further critical point $a$ of multiplicity $p$, this would have to be the only finite critical point, since the sum of ``multiplicities minus $1$" of all finite critical points needs to sum up to $\deg(f')=p-1$. 
In other words, $f$ is then unicritical and postcritically infinite, whence the assertion follows from Case 1) of Theorem \ref{thm:unicrit}.

 In all other cases, the multiplicities of all finite critical points are smaller than (and hence, coprime to) $p=\deg(f)$. The assertion is now immediate from Theorem \ref{thm:main}.
 \end{proof}

 \begin{proof}[Proof of Corollary \ref{cor:odddeg}]
 Any root of $g$ is a root of odd multiplicity of $f'$ and hence has even multiplicity under $f$. Since $\deg(f)$ is odd, the assertion follows immediately from Theorem \ref{thm:main}.
 \end{proof}
 
 \begin{proof}[Proof of Corollary \ref{cor:trinom}]
 a) is immediate from Theorem \ref{thm:main}, since here $a$ has multiplicity $k$. For b), calculation of the derivative shows that the critical points are exactly the $\gamma_i$ ($i=0,\dots, d-k$). These are furthermore pairwise distinct by the assumptions on the parameters, and their multiplicities under $f$ are $k(>1)$ for $\gamma_0$ and $2$ for all other $\gamma_i$. Since these multiplicities are coprime to $d$ by assumption, the assertion follows again from Theorem \ref{thm:main}.
 \end{proof}

  After deriving all stated consequences, it remains to prove Theorem \ref{thm:tech_strong}.
  \begin{proof}[Proof of Theorem \ref{thm:tech_strong}]
 The assertion will follow from combination of Theorem \ref{thm:spec_in} and Lemma \ref{lem:betal}. We note that the result could essentially be read out of the proof of \cite[Theorem 5]{Betal}, which concludes the existence of infinitely many ramified primes in the splitting fields of  $f^{\circ n}(X)-a$, $n\in \mathbb{N}$, without drawing any conclusions about ramification indices (but with such conclusions being possible via obvious modification of the argument).

{\it Step 1: Basic preparations}. 
First note that we may replace $K$ by a finite extension $F\supseteq K$. Indeed, let $K_n$ denote the splitting field of (the numerator of) $f^{\circ n}(X)-a$ over $K$. If, as we will do, we can verify the existence of elements of certain orders in $\textrm{Gal}(F\cdot K_n/F\cdot K_{n-1})$ (i.e., the kernel of the projection $\pi_n: G_n\to G_{n-1}$ after base change from $K$ to $F$), the same will follow for $\textrm{Gal}(K_n/K_{n-1})$, since the latter contains the former as a subgroup. We may and will therefore assume that all critical points of $f$ are $K$-rational. 
We will furthermore assume for simplicity that $a=0$ and $f(\infty)=\infty$.\footnote{In particular, $0$ is then not postcritical for $f$, as required to apply Lemma \ref{lem:betal}.} This is without loss of generality, up to possibly extending the field of definition of $f$, due to the following: Let $b$ be any root of $f(X)-X$ (assumed to lie in $K$ upon finite base change). We have $b\ne a$, since $a$ was assumed not to be a fixed point of $f$. 
Pick a $K$-rational M\"obius transformation $\mu$ with $\mu(0)=a$ and $\mu(\infty)=b$, and set $\tilde{f}=\mu^{-1}\circ f\circ \mu$. Then $\tilde{f}(\infty)=\infty$, and furthermore the splitting fields of $f^{\circ n}(X)-a$ and of $\tilde{f}^{\circ n}(X)$ coincide for all $n$.

{\it Step 2}. 
We will construct elements of order divisible by $q$ in $\ker(\pi_n)$ for infinitely many $n$ as inertia group generators at suitable primes in the splitting field $K_n/K$ of $f^{\circ n}(X)$. To this end, pick, for a prime $p$ of $K$, the smallest $n:=n(p)\in \mathbb{N}$ such that  $\nu_p(f^{\circ n-1}(t_i))$ is positive and coprime to $q$, as long as such $n$ exists. Lemma \ref{lem:betal} guarantees the existence of infinitely many such primes $p$ as $n$ runs through all natural numbers. We may of course assume, via exempting finitely many $p$, that $p$ is at most tamely ramified in $K_{n(p)}/K$. We will show that $p$ ramifies of ramification index divisible by $q$ in $K_{n(p)}/K$. Conversely, by finiteness of the set of ramified primes in a finite extension, we may of course iteratively pick $p$ such that it does not ramify in any of the fields $K_{n(p')}$ for any of the finitely many previously picked primes $p'$, and will thus indeed obtain an infinite sequence of integers $n$ with the desired properties. We may assume that the ramification index at $p$ in the splitting field $K_{n-1}$ of $f^{\circ n-1}(X)$ is coprime to $q$, or else the minimal $k<n$ such that the ramification index at $p$ in $K_k$ is divisible by $q$ yields an element of order divisible by $q$ in the kernel $\ker(\pi_k)$ (just raise the inertia group generator at $p$ to the power its order modulo $\ker(\pi_k)$). 

Write the rational function $f$ as  \begin{equation}
\label{eq1}
f(X)=\frac{f_1(X)}{f_2(X)}\end{equation} with coprime  polynomials $f_1,f_2\in K[X]$, and recall that $\deg(f_2)<\deg(f_1)$ by our assumption $f(\infty)=\infty$.
We may then divide by the leading coefficient of $f_1(X)$ to assume $f_1(X)-tf_2(X)\in K[t][X]$ is monic. Upon excluding finitely many primes $p$, we may assume that $p$ is a good prime for $f_1(X)-tf_2(X)$. In particular, via Condition i) in Theorem \ref{thm:spec_in}, this means that all coefficients of $f_1$ and $f_2$ lie in the local ring $R_{(p)}$ of $p$, and via Condition iii), so does $t_i$. 
Let $\mathfrak{p}$ be a prime of $K_{n-1}$ extending $p$. As in the proof of Lemma \ref{lem:betal}, homogenize Equation \eqref{eq1} as $f(X/Y)=\frac{F_1(X,Y)}{G_1(X,Y)}$, and similarly \begin{equation}
\label{eq2}
f^{\circ j}(X/Y)=\frac{F_j(X,Y)}{G_j(X,Y)}
\end{equation}
with suitable homogeneous polynomials $F_j,G_j\in R_{(p)}[X,Y]$ for all $j\in \mathbb{N}$.

We may assume, as in the proof of Lemma \ref{lem:betal}, that $p$ does not divide the resultant of $F_1$ and $G_1$,
 and hence, since $F_j(X,Y)=F_1(F_{j-1}(X,Y), G_{j-1}(X,Y))$ and $G_j(X,Y)= G_1(F_{j-1}(X,Y), \\ G_{j-1}(X,Y))$, that $\nu_p(G_{n-1}(t_i,1))=0$ and thus $\nu_p(f^{\circ n-1}(t_i)) = \nu_p(F_{n-1}(t_i,1))$. 
Then write $F_{n-1}(X,Y)=\prod_{j=1}^N (X-\delta_jY)$ over $K_{n-1}$. Here, the $\delta_j$ are all $\mathfrak{p}$-integral, since $F_{n-1}$ is inductively seen to be monic in $X$. 
We have $\sum_{j=1}^N \nu_{\mathfrak{p}}(t_i-\delta_j) = \nu_{\mathfrak{p}}(f^{\circ n-1}(t_i)) = \alpha\cdot \nu_p(f^{\circ n-1})(t_i)$, where $\alpha\in \mathbb{N}$ is coprime to $q$, since the ramification index of $\mathfrak{p}$ over $p$ is coprime to $q$. By assumption, the right side of the equality is coprime to $q$, hence $\nu_{\mathfrak{p}}(t_i-\delta_j)$ is coprime to $q$ (and also positive, by $\mathfrak{p}$-integrality of the argument) for at least one $j$. 
Consider the splitting field $L_j$ of $f(X)-\delta_j$ over $K_{n-1}$. This clearly embeds into $K_n$ (since $f^{\circ n}(\delta_j)=0$). Moreover, $\mathfrak{p}$ can be assumed to be good for $f(X)-t$ over $K_{n-1}$ by Corollary \ref{cor:base_change}, since $p$ is good for $f(X)-t$ over $K$. Thus, by Theorem \ref{thm:spec_in}, $\mathfrak{p}$ ramifies in $L_j/K_{n-1}$ of ramification index $\frac{|I_{t_i}|}{\textrm{gcd}(|I_{t_i}|, \nu_{\mathfrak{p}}(t_i-\delta_j))}$, with the numerator, but not the denominator of the latter fraction being divisible by $q$. This concludes the proof.
  \end{proof}

  We give an analog of Theorem \ref{thm:tech_strong} working under weaker assumptions, namely {\it not} requiring a point of multiplicity  divisible by a prime not dividing $\deg(f)$. This relaxation necessitates some extra ideas in the proof, and 
  unlike before, we will not be able to derive a conclusion about elements of $G_n:=\textrm{Gal}(f^{\circ n}(X)-a/K)$ fixing {\it all} roots of $f^{\circ n-1}(X)-a$ (i.e., lying in the kernel of $\pi_n: G_{n}\to G_{n-1}$), but only about elements fixing at least {\it one} root (i.e., lying in a a certain block stabilizer of the imprimitive permutation group $G_n$). 
 This will only allow us to invoke part b), not part a) of the group theoretical result Corollary \ref{cor:stabpr}, which may seem less attractive. Nevertheless, it enables us to cover some new classes of polynomials, such as the ones in Theorem \ref{thm:main_extra}.
  \begin{theorem}
  \label{thm:tech_stronger}
  Let $f\in K(X)$ and let $a\in K$ be neither postcritical nor a fixed point of $f$. 
 Let $t_i$ be a wandering critical point of multiplicity $d$ under $f$, and let $q$ be a prime power dividing $d$. Let  $G:=\mathrm{Gal}(f(X)-t/K(t))$, and let $\sigma_i\in G$ be an inertia group generator over $t\mapsto t_i$ in the splitting field of $f(X)-t$.  Then for infinitely many $n\in \mathbb{N}$, there exists a root $\alpha_{n-1}$ of $f^{\circ n-1}(X)-a$ such that the group $H_n:=\textrm{Gal}(f(X)-\alpha_{n-1}/K(\alpha_{n-1}))$ contains an element $y\in \langle \sigma_i\rangle$ (up to conjugation in $G$), of order $\textrm{ord}(y)$ divisible by $q$.
  \end{theorem}
  \begin{proof}
  We keep the assumptions on $f$ and $K$ (including in particular the assumption $a=0$) and the notation from the setup in Step 1 of the proof of Theorem \ref{thm:tech_strong}. Pick $n\in\mathbb{N}$ and a good prime $p$ for $f(X)-t$
  such that $\nu_p(f^{\circ n-1}(t_i))$ is positive and coprime to $q$. 
  Let $\mathfrak{p}$ be a prime of $K_{n-1}$ extending $p$, and factorize the numerator of $f^{\circ n-1}(X/Y)$ in the expression from Equation \eqref{eq2} as \begin{equation}
  \label{eq3}
  \prod_{j=1}^N (X-\delta_jY)\in K_{n-1}[X,Y].
  \end{equation}
  Let $I_\mathfrak{p}$ be the inertia group at $\mathfrak{p}$ in $K_{n-1}/K$ and $M\subset K_{n-1}$ the fixed field of $I_{\mathfrak{p}}$. Let $\tilde{\mathfrak{p}}$ be the prime of $M$ which is extended by $\mathfrak{p}$, so that $\tilde{\mathfrak{p}}/p$ is unramified and $\mathfrak{p}/\tilde{\mathfrak{p}}$ is totally ramified, say, of degree $h$. In particular, for any $x\in K$, we have $\nu_p(x)=\nu_{\tilde{\mathfrak{p}}}(x)$, and $\nu_{\mathfrak{p}}(x) = h\cdot\nu_p(x)$.
In particular, if $\rho_1(X,Y),\dots, \rho_R(X,Y)\in M[X,Y]$ denote the irreducible factors of \eqref{eq3} over $M$, we have 
$$\sum_{j=1}^{R} \nu_{\tilde{\mathfrak{p}}}(\rho_j(t_i,1))= \nu_{\tilde{\mathfrak{p}}}(f^{\circ n-1}(t_i)) = \nu_{p}(f^{\circ n-1}(t_i)),$$ which by assumption is positive and coprime to (the prime power) $q$. There is therefore at least one index $j_0\in \{1,\dots, R\}$ such that $\nu_{\tilde{\mathfrak{p}}}(\rho_{j_0}(t_i,1))$ is 
coprime to $q$.
Now factorize further $\rho_{j_0}(X,Y) = \prod_{k=1}^\ell (X-\delta_kY)$ (with the appropriately numbered roots $\delta_1,\dots, \delta_\ell\in K_{n-1}$).

Set $M':=M(\delta_1)$; i.e., $M'$ is a cyclic degree-$\ell$ extension of $M$, totally ramified at $\tilde{\mathfrak{p}}$. Let $\widehat{\mathfrak{p}}$ be the unique extension of $\tilde{\mathfrak{p}}$ in $M'$. Since the elements of $\textrm{Gal}(M'/M)$ fix $\widehat{\mathfrak{p}}$ and $t_i$ while permuting $\delta_1,\dots, \delta_\ell$ transitively, all the valuations $\nu_{\widehat{\mathfrak{p}}}(t_i-\delta_k)$, $k=1,\dots, \ell$ are identical, and since furthermore $\sum_{k=1}^\ell \nu_{\widehat{\mathfrak{p}}}(t_i-\delta_k) = \nu_{\widehat{\mathfrak{p}}}(\rho_{j_0}(t_i,1)) = \ell\cdot \nu_{\tilde{\mathfrak{p}}}(\rho_{j_0}(t_i,1))$, the valuations $\nu_{\widehat{\mathfrak{p}}}(t_i-\delta_k)$ must after all be (positive and) coprime to $q$.

Consider the splitting field $L$ of $f(X)-\delta_1$ over $M'$.  The prime $\widehat{\mathfrak{p}}$ can be assumed to be good for $f(X)-t$ over $M'$ by Corollary \ref{cor:base_change}, since $p$ is good for $f(X)-t$ over $K$. Thus, by Theorem \ref{thm:spec_in}, $\widehat{\mathfrak{p}}$ ramifies in $L/M'$, with inertia group equal to a subgroup of $I_{t_i}$ (up to conjugacy), of order $\frac{|I_{t_i}|}{\textrm{gcd}(|I_{t_i}|, \nu_{\widehat{\mathfrak{p}}}(t_i-\delta_1))}$. Since the numerator is divisible by $q$ whereas the denominator is coprime to $q$, we have identified a subgroup of $I_{t_i}$ of order divisible by $q$ inside $\textrm{Gal}(L/M')$. Setting $\alpha_{n-1}:=\delta_1$, the group $\textrm{Gal}(L/M')$ embeds naturally as a subgroup into $H_n$ (via restriction to the splitting field of $f(X)-\alpha_{n-1}$ over $K(\alpha_{n-1})$), from which the assertion follows.
  \end{proof}

  \begin{proof}[Proof of Theorem \ref{thm:main_extra}]
  Let $E/K(t)$ be the splitting field of $f(X)-t$. By assumption, the inertia group over $t\mapsto b$ in $E/K(t)$ is generated by a single cycle $\sigma$ of (odd) length $k$ (since this is the multiplicity of the only critical point $a$ mapped to $b$ by $f$. Let $n:=\deg(f)$. Since $k$ is assumed to not divide $n$, there exist an (odd) prime power $q=p^e$ dividing $k$, but not dividing $n$ (i.e., the $p$-adic valuation of $n$ is smaller than $e$). Note that the elements of order $q$ in $\langle \sigma \rangle$ are products of an odd number of disjoint $q$-cycles.
  Due to Theorem \ref{thm:tech_stronger} and Corollary \ref{cor:stabpr}b),
   it suffices to show that  such elements cannot be contained in an iterated wreath product $AGL_1(p_r)\wr\dots\wr AGL_1(p_1)$, where $p_1,\dots, p_r$ are primes with $p_1\cdots p_r=n$. Let $d$ be the $p$-adic valuation of $n$. Since $q=p^e$ does not divide $n$, we have $d<e$. Let $\sigma$ be an element of $AGL_1(p_{r})\wr\dots\wr AGL_1(p_1)$ all of whose cycle lengths are $1$ or $q$. We will show that $\sigma$ must necessarily have an even number of $q$-cycles; the assertion obviously follows from this claim. 
  %
  
  As outlined in Section \ref{sec:prelim_wreath}, one has a sequence of partitions $\mathcal{S}_i:=(\lambda_{i,1},\dots, \lambda_{i,N_i})$ for each $i\in \{1,\dots, r\}$, all corresponding to cycle types in $AGL_1(p_i)$; and, for every cycle of $\sigma$, a unique sequence  $((a_1,b_1),\dots, (a_r,b_r))$ with $a_i\in \{1,\dots, N_i\}$ and $b_i\in \{1,\dots, \#\lambda_{i,a_i}\}$  identifying the elements of $\mathcal{S}_i$, $i=1,\dots, r$, used to compose the respective cycle. For each $q$-cycle $c$ of $\sigma$, denote by $i(c)$ the lowest index $i$ for which the $b_i$-th part of $\lambda_{i,a_i}$ is larger than $1$ (i.e., among the block systems induced by the wreath product, $i(c)$ identifies the level of the ``coarsest" block system on which $c$ acts nontrivially). We define an equivalence relation on the set of $q$-cycles of $\sigma$, calling two cycles equivalent if they share the same $(i(c), a_{i(c)})$. We claim that each equivalence class is of even cardinality. For this, note first that, for each cycle $c$, $\lambda_{i(c), a_{i(c)}}$ must be an element of $AGL_1(p_{i(c)})$ {\it other} than a $p_{i(c)}$-cycle, and hence, have a unique fixed point (and $p$-cycles otherwise). Indeed, assume on the contrary that $\lambda_{i(c),a_{i(c)}}$ is a full cycle. Since $q$ does not divide $n=p_1\cdots p_r$, there must be an index $j>i(c)$ for which $\lambda_{j,a_j}$ is {\it not} a full cycle, i.e., has a fixed point. Due to this fixed point, however, every such $\lambda_{j,a_j}$ does {\it not} increase the minimal cycle length of the partition it extends and since this partition was assumed to factor through a full $p$-cycle, this minimal cycle length is $>1$. All such cycle lengths need to be eventually extended into $q(=p^e)$-cycles in $\sigma$, which is thus possible only be using a $p$-cycle inside $AGL_1(p)$ a total of $e$ times. This is however impossible, since due to $p^e$ not dividing $n$, there are not sufficiently many $p_i$ equal to $p$.
   
  We have thus shown that for each $q$-cycle $c$ of $\sigma$, $\lambda_{i(c),a_{i(c)}}$ is a partition of $p_{i(c)}$ of the form $(1,p,p,\dots, p)$. This enforces $p_{i(c)}\equiv 1$ mod $p$, and hence even $p_{i(c)}\equiv 1$ mod $2p$ (since both $p$ and $p_{i(c)}$ are odd primes), meaning that the number of $p$-cycles in $\lambda_{i(c), a_{i(c)}}$ is necessarily even. Furthermore, by definition of $i(c)$, the cycle $c$ necessarily extends some $p$-cycle of this partition, and conversely each such $p$-cycle
   lifts only to cycles of length $q$ of $\sigma$, and hence necessarily to the {\it same} number of such $q$-cycles 
  (namely, $(p_{i(c)+1}\cdots p_{i_r})/(q/p)$ of them). Thus, the number of $q$-cycles sharing the same $(i(c), a_{i(c)})$ must be even, proving the claim and completing the proof of the theorem.
  \end{proof}

  Using (the so-far unused) Lemma \ref{lem:betal}b) gives a strengthening of Theorem \ref{thm:tech_stronger}, allowing to deduce ``density zero" results even in some cases  where every ramification index divides the degree $\deg(f)$. We give one particular example.
  \begin{lemma}
  \label{lem:more_ex} Let $f=(X-a)^k(X-b)^mg(X)+c\in K[X]$, where $(X-a)(X-b)g(X)$ is separable, 
   $k$ and $m$ are prime numbers with $k\le m-2$, and $c$ is a wandering point of $f$. Then the set of stable primes of $f$ is of density zero.
   \end{lemma}
   \begin{proof} The inertia group at $t\mapsto c$ in the splitting field of $f(X)-t$ is generated by an element of cycle type $(k.m.1^{\deg(g)})$. If one of $k,m$ does not divide $\deg(f)$, we may simply apply Theorem \ref{thm:main}, so assume $km|\deg(f)$. From 
  Lemma \ref{lem:betal}b), we derive the existence of infinitely many $n\in \mathbb{N}$ such that $\mathrm{gcd}(\{km\}\cup \{\nu_p(f^{\circ n}(t_i)): p \text{\ ``good" for $f(X)-t$ over $K$} \text{ and } \nu_p(f^{\circ n}(t_i)) > 0\}) = 1$; hence this set of valuations $\nu_p$ contains at least one value coprime to $k$ and one coprime to $m$.
  Carrying out the proof of Theorem \ref{thm:tech_stronger} for both of these primes individually, we derive the existence of infinitely many $n\in \mathbb{N}$, and for each such $n$ the existence of a root $\alpha$ of $f^{\circ n-1}(X)$ such that the stabilizer of $\alpha$ in $\textrm{Gal}(f^{\circ n}(X)/K)$ (i.e., the block stabilizer of some block induced by the decomposition $f^{\circ n} = f^{\circ n-1}\circ f$) contains an element acting as a $k$-cycle on the roots of $f(X)-\alpha$, and another one acting as an $m$-cycle.\footnote{That we may, without loss of generality, choose the {\it same} root $\alpha$ for both $k$ and $m$ is due to the fact that $f^{\circ n-1}(X)$ may be assumed irreducible - or else there are no stable primes - and hence stabilizers of roots are pairwise conjugate.}  For this to be compatible with a positive density set of stable primes, we must find a transitive subgroup $H\le AGL_1(p_r)\wr\dots\wr AGL_1(p_1)$, with $\prod_{i=1}^r p_i = \deg(f)$, such that $H$ contains both a $k$-cycle and an $m$-cycle. Building such a cycle via composing partitions as explained in Section \ref{sec:prelim_wreath} and recalling that $k$ and $m$ are prime numbers, it becomes obvious that these cycles can only arise from a cycle of the respective length on the ``uppermost level", i.e., from (the restriction to some block of) an element in the kernel of the projection to  $AGL_1(p_{r-1})\wr\dots\wr AGL_1(p_1)$; in other words, $AGL_1(p_r)$ must contain both of these cycles. This is however impossible, since $k\le m-2$ and $AGL_1(p_r)$ does not contain a cycle with more than one fixed point.
  \end{proof}

Note once again that our results use the (unconditionally known) existence of infinitely many prime divisors (with certain extra conditions) in dynamical sequences $(f^{\circ n}(a))_{n\in \mathbb{N}}$. Conditionally on the abc conjecture, much more is known. E.g., under very weak assumptions on $f$, squarefree primitive prime divisors (i.e., dividing the $f^{\circ n}(a)$ exactly once, but dividing no previous term $f^{\circ k}(a)$, $k\in \mathbb{N}$) actually exist not only for infinitely many, but for all but finitely many $n$, cf.\ \cite{GNT}. If this could be shown to hold unconditionally, it would enable to deal with numerous further classes of polynomials regarding the mod-$p$ stability problem considered here, notably since this would open up working with not just one inertia group generator (coming from some critical value $a$ of $f$), but with several ``at a time" (i.e., for one and the same level $n$ of the iteration). We will refrain from carrying out any such analysis in detail, since the present work was deliberately aimed at obtaining unconditional results.

\appendix

 \section{A result for ``large" arboreal representations}
   \label{sec:large}   
 We give some further evidence in favor of Conjecture \ref{conj:bold} via the following observation, which implies that, in order for a postcritically infinite polynomial $f$ and any $a\in K$ to have a positive density set of stable primes, the Galois group needs to be ``much smaller than expected", and in particular of infinite index inside the corresponding dynamical monodromy group. Since this is essentially a result about monodromy groups, it is more straightforward to derive than our main results because combination with the specialization results Theorem \ref{thm:spec_in} and Lemma \ref{lem:betal} is not required.
  \begin{lemma}
  \label{lem:cond}
  Let $f\in K[X]$ be a postcritically infinite polynomial, let 
  $G_{t,n}:=\mathrm{Gal}(f^{\circ n}(X)-t/K(t))$ and $G_{t,\infty}:=\varprojlim_n G_{t,n}$. Similarly, for $a\in K$ let $G_{a,n}:=\mathrm{Gal}(f^{\circ n}(X)-a/K)$ and $G_{a,\infty}:=\varprojlim_n G_{a,n}$. Then there exists a constant $d>1$ (depending only on $\deg(f)$) such that if
  $[G_{t,n}: G_{a,n}] = o(d^n)$ (as a function in $n$),  
  then the density of stable primes of $(f,a)$ is zero. In particular, the conclusion holds as soon as $[G_{t,\infty}:G_{a,\infty}]<\infty$.
  \end{lemma}
  \begin{proof}
  We claim that for all $f$ in question, the proportion of full cycles (i.e., generators of cyclic transitive subgroups) in $G_{t,n}:=\mathrm{Gal}(f^{\circ n}(X)-t/K(t))$ is at most $d^{-n}$ (with $d>1$ yet to be chosen). 
  This then proves the assertion, since indeed, it then instantly follows that the proportion of full cycles in any sequence of subgroups of $G_{t, n}$ ($n\in \mathbb{N}$) of index $o(d^n)$ is $o(1)$, i.e., converges to $0$ as $n\to \infty$. But if $G_{a,n}$, $n\in \mathbb{N}$ is such a sequence, then the proportion of full cycles equals the density of primes modulo which $f^{\circ n}(X)-a$ is irreducible. 
  
  We show the claim using Proposition \ref{prop:stabpr_gp}. 
  %
 Write $f = f_1\circ\dots\circ f_r$ as a composition of indecomposable polynomials, and for convenience, define the sequence $(g_n)_{n\in \mathbb{N}}$ as $g_{kr+i}=f_i$ for all $k\ge 0$ and $i\in \{1,\dots, r\}$, so that $(g_1\circ\dots\circ g_{kr})(X)-t = f^{\circ k}(X)-t$. Let $K_n$ be the splitting field of $(g_1\circ\dots\circ g_n)(X)-t$ over $K(t)$. By the indecomposability assumption, $\textrm{Gal}(K_n/K_{n-1})$ embeds into a direct power of a primitive group, and its projection to any single component contains a cyclic transitive subgroup (namely, a suitable power of the inertia group generator at $t\mapsto \infty$). Let $c_n$ denote the proportion of full cycles among all elements of $\textrm{Gal}(K_n/K(t))$. From Proposition \ref{prop:stabpr_gp}, we have $c_n\le \delta c_{n-1}$ for a constant $\delta<1$ depending only on $\deg(f)$, unless $\textrm{Gal}(K_n/K_{n-1})$ is an elementary-abelian group (say, of exponent $p_n$) of smaller than maximal-possible order, i.e., not equal to the full direct power $(C_{p_n})^{\deg(g_1\circ\dots\circ g_{n-1})}$. We will show that the latter condition is violated for at least one $n\in \{kr+1,\dots, (k+1)r\}$ (for every $k\in \mathbb{N}$), which directly implies the claim with $d:=\delta^{-1/r}$, thus completing the proof. 
 To this end, recall that $f$ was assumed to be postcritically infinite, and thus, for every $k\in \mathbb{N}$, there is at least one critical point $\alpha$ of $f$ such $f^{\circ k}(\alpha)$ is not in $\{f^{\circ i}(\beta)\mid \beta \text{ a critical point of } f; i<k\}$. This translates to saying that $f^{\circ k}(\alpha)$ is a branch point of $x\mapsto f^{\circ k}(x)$, but not of any $x\mapsto f^{\circ i}(x)$  with $i<k$. Hence, there is at least one $n\in \{kr+1,\dots, (k+1)r\}$ such that $(g_1\circ \cdots \circ g_n)(\alpha)$ is a branch point of $K_n/K(t)$, but not of $K_{n-1}/K(t)$. 
 Then the inertia group generator $\sigma$ at $t\mapsto (g_1\circ \cdots \circ g_n)(\alpha)$ in $K_n/K(t)$ lies in the kernel of the projection to $\textrm{Gal}(K_{n-1}/K(t))$. By Proposition \ref{prop:stabpr_gp} (case a)), the claim follows as soon as $\deg(g_n)$ is non-prime, so assume $\deg(g_n)=:p_n$ is a prime. But still, by Proposition \ref{prop:stabpr_gp} (case a) and b)), the claim follows unless $\sigma$ acts trivially or as a $p_n$-cycle on every component. So $g_n$ must be a polynomial of degree $p_n$ with a finite critical point of multiplicity $p_n$ under $g_n$, and this is hence the {\it only} finite critical point of $g_n$. So  
 $\sigma$ is in fact nontrivial only on a single component. But then, due to the transitive action of $\textrm{Gal}(K_n/K(t))$ on the components, it follows easily that $\textrm{Gal}(K_n/K_{n-1})$ contains the full direct power $C_{p_n}^{\deg(g_1\circ\dots\circ g_{n-1})}$, and again, the claim follows by Proposition \ref{prop:stabpr_gp} (case c)). 
  \end{proof}

  \begin{remark}
  \label{rem:ferr}
  Compare Lemma \ref{lem:cond} to \cite[Theorem 2.3]{Ferr}, which obtained the density zero conclusion under the assumption that the index of $G_{a,n}$ inside the maximal possible group, namely inside the $n$-fold iterated wreath product $[S_d]^n$ of $S_d$ ($d:=\deg(f)$) is $o(d^n)$. This ``largeness" assumption also implies that $f$ is postcritically infinite, but is far from equivalent to it.
  This is because a group with this largeness property automatically has Hausdorff dimension $\lim_{n\to \infty} \frac{\log|G_{a,n}|}{\log|[S_d]^n|} = 1$, whereas on the contrary, for every $f$ possessing some iterate $F:=f^{\circ n}$ with $\textrm{Gal}(F(X)-t/K(t)) < [S_d]^n$ (which is certainly the case for PCF maps due to their dynamical monodromy group being topologically finitely generated), the Hausdorff dimension is easily seen to be $<1$.
  \end{remark}
  
 As already mentioned, it is generally expected, although hard to prove, that the assumption of Lemma \ref{lem:cond} holds (even with the ``slowly growing" index replaced by finite index) outside of some very concrete special cases (see, e.g., \cite{BDGHT}).

  \section{Some remarks on low degree polynomials} 
  \label{sec:lowdeg}
Motivated by the preceding observations, we may ask about the {\it maximal} possible proportion of stable primes for polynomials of a given degree, and about the exact polynomials which attain this maximum. Due to our above results, we can completely answer this problem for cubic polynomials over $\mathbb{Q}$.
  \begin{theorem}
  \label{thm:max}
  Let $f\in \mathbb{Q}[X]$ be a cubic polynomial, $a\in \mathbb{Q}$, and denote the proportion of stable primes of $(f,a)$ by $\pi_{f,a}$. Then $\pi_{f,a}\le \frac{2}{3}$, with equality if and only if $(f,a)$ fulfill all of the following, up to conjugation\footnote{I.e., replacing $(f,a)$ by $(\mu\circ f\circ \mu^{-1}, \mu(a))$ for a linear $\mu\in \mathbb{Q}[X]$.} over $\mathbb{Q}$.
  \begin{itemize}
  \item[a)] $f(X) = \pm (X^3-3cX)$, where $c$ is a positive integer of the form $c=\alpha^2+3\beta^2$ for $\alpha,\beta\in \mathbb{Z}$.
  \item[b)] $a\notin f(\mathbb{Q})$, but $a\in g(\mathbb{Q})$, where 
  $g(X):=\frac{2c(\alpha X^2+6\beta X-3\alpha)}{X^2+3}$ with $\alpha,\beta, c$ as in a).
  \end{itemize}
  \end{theorem}
  \begin{remark}
  The polynomials $f$ above are the normalized third Dickson polynomials with parameter $c$, as well as their negatives. In particular, $c=1$ gives the third Chebyshev polynomial.
  \end{remark}
  \begin{proof}
  The bound $\pi_{f,a}\le \frac{2}{3}$ is trivial; indeed, $\pi_{f,a}$ is upper-bounded by the proportion of full cycles in $\textrm{Gal}(f^{\circ n}(X)-a/\mathbb{Q})$ for any given $n\in \mathbb{N}$, and the bound $2/3$ already follows from $n=1$.
  To see that the pairs $(f,a)$ fulfilling the above two conditions indeed attain the bound, note that due to $\pm f$ being conjugate to a Dickson  polynomial, one has $G_{t,n}:=\textrm{Gal}(f^{\circ n}(X)-t/\mathbb{Q}(t)) \cong C_{3^n}\rtimes \textrm{Aut}(C_{3^n})$. It is well-known  (and easy to show more generally with $3$ replaced by any odd prime) that any preimage of a $3$-cycle in $G_{t,1}\cong S_3$ is a full cycle in $G_{t,n}$. In particular, the proportion of full cycles in $G_{t,n}$ is $\frac{1}{3}$ for all $n$. Note also that, due to $f$ being PCF, only finitely many primes ramify in the compositum of all splitting fields of $f^{\circ n}(X)-a$, $n\in \mathbb{N}$; cf.\ \cite{Betal}. We may thus invoke Chebotarev's density theorem to see that in order to get the desired proportion $2/3$, it is necessary and sufficient that $\textrm{Gal}(f(X)-a/\mathbb{Q})\cong C_3$. It is now elementary to parameterize the quadratic subfield of the splitting field of $f(X)-t$:
  it is generated by the root of the discriminant of the latter polynomial, and thus easily calculated as $\mathbb{Q}(t)(\sqrt{12c^3-3t^2})$; indeed, all observations so far hold for any $c\in \mathbb{Z}\setminus\{0\}$. It then remains to find out when the conic $12C^3-3T^2-Y^2=0$ has rational points, and if so, to parameterize it. Existence of points is very easily seen to be equivalent to $c$ being of the form $\alpha^2+3\beta^2$, while calculation of the parameter $t=g(X)$ is an elementary exercise whose details we omit.

  We have thus completely analyzed the case of polynomials which are conjugate over $\overline{\mathbb{Q}}$ to the Chebyshev polynomial (resp., its negative), since any such polynomial is conjugate over $\mathbb{Q}$ to a Dickson polynomial (resp., its negative). 
  To see that no other cubic polynomial can attain the bound $\frac{2}{3}$, note first that Corollary \ref{thm:primedeg} rules out all postcritically infinite  cubics $f$ other than the ones linearly related to $X^3$. On the other hand, for any $f$ linearly related to $X^3$, all full cycles must be contained in the index-$2$ subgroup fixing $\zeta_3$, i.e., one has full cycle proportion $\le \frac{1}{3}$ (even already in $G_{a,1}$). 
   This rules out any $f$ linearly related to $X^3$. To deal with the remaining PCF cases, it is useful to note that, by Proposition \ref{prop:stabpr_gp}b), to attain the bound $\frac{2}{3}$, one necessarily needs $G_{a,n}$ to be a $3$-group for all $n$, and therefore in particular all roots of $f^{\circ n}(X)-a$ to be real. One may now use the classification result obtained in \cite{AMT}\footnote{A little bit of care is required since this classification is up to conjugation over $\overline{\mathbb{Q}}$, whereas a priori only conjugation over $\mathbb{R}$ preserves real fibers. This problem is, however, easily solved upon considering the concrete ramification types of iterates, which can be read out of \cite[Table 1]{AMT}, and noting, e.g., that a real fiber can only occur for $a$ in between two successive real critical values, both of ramification index $2$.} and verify directly that no such cubic $f$, other than the ones conjugate to the Chebyshev polynomial or its negative, has any values $a\in \mathbb{Q}$ rendering $f(f(X))-a$ totally real. (Note that this is a short finite computation, since the number of real roots of $g(X)-a$ for a polynomial $g\in \mathbb{R}[X]$ is constant for $a$ in a fixed component of $\mathbb{R}$ minus the critical values.)
  \end{proof}
  \begin{remark}
  \begin{itemize}
  \item[a)] Using the results of \cite{MOS}, which show in particular that postcritically infinite quadratic polynomials necessarily have a density zero set of stable primes, as well as the known short list $\{X^2-c\mid c\in \{0,1,2\}\}$ of PCF quadratics over $\mathbb{Q}$ up to conjugacy, 
  one may analogously list all quadratic polynomials $f$ and values $a\in \mathbb{Q}$ obtaining the maximal possible value $\pi_{f,a}=\frac{1}{2}$ for this degree. A brief computation shows that again, up to conjugacy, the Chebyshev polynomial (i.e., $f=X^2-2$) is the only polynomial admitting such values.
  \item[b)] We conjecture that the value $\frac{2}{3}$ is the largest possible value for $\pi_{f,a}$ for $f\in \mathbb{Q}[X]$ of {\it any} degree $>1$. Note that the degree-$p$ monomials $X^p$, resp., Chebyshev polynomials $T_p$ (for $p>3$ prime), are not suitable as composition factors for a hypothetical ``record breaking" polynomial. Indeed, including them can at best lead to $\pi_{f,a}=\frac{1}{p}$, resp., $\pi_{f,a}=\frac{2}{p}$, since their splitting field  always contains the extension $\mathbb{Q}(\zeta_p)/\mathbb{Q}$, resp., its maximal real subextension, meaning that all full cycles are contained inside a fixed normal subgroup of index $p-1$, resp., index $\frac{p-1}{2}$.
  \end{itemize}
  \end{remark}

  \end{document}